\documentclass[english,oneside]{smfart}
\PassOptionsToPackage{dvipsnames,svgnames}{xcolor}

\newcommand{\opn}[1]{\operatorname{#1}}

\usepackage{amsthm,mathtools,amssymb,textcomp,extarrows,bm,mleftright,graphicx,stmaryrd,scalerel,tikz,enumitem,float,booktabs,fontspec,mathspec}
\usepackage[dvipsnames,svgnames]{xcolor}
\DeclareMathAlphabet{\mathbbb}{U}{bbold}{m}{n}

\usetikzlibrary{cd}
\mleftright

\makeatletter
\AtBeginDocument{%
 \let\glb@currsize\relax
}
\makeatother

\usepackage{mathrsfs}

\ExplSyntaxOn
\NewDocumentCommand{\definealphabet}{mmmm}
 {%
  \int_step_inline:nnn { `#3 } { `#4 }
   {
    \cs_new_protected:cpx { #1 \char_generate:nn { ##1 }{ 11 } }
     {
      \exp_not:N #2 { \char_generate:nn { ##1 } { 11 } }
     }
   }
 }
\ExplSyntaxOff
\definealphabet{bb}{\mathbb}{A}{Z}
\definealphabet{cal}{\mathcal}{A}{Z}
\definealphabet{frak}{\mathfrak}{A}{z}
\definealphabet{rm}{\mathrm}{A}{z}
\definealphabet{bf}{\mathbf}{A}{z}
\definealphabet{scr}{\mathscr}{A}{Z}
\definealphabet{tt}{\mathtt}{A}{Z}
\definealphabet{ul}{\underline}{A}{z}

\newcommand{\lto}{\longrightarrow}

\newcommand{\wtilde}{\widetilde}

\newfontface\logix {logix.otf}[Scale=1.0,NFSSFamily=logix]
\DeclareSymbolFont{Logix}{TU}{logix}{m}{n}
\newcommand{\lgl}[1]{\ensuremath{\Udelimiter 4 \symLogix "#1}}

\newcommand{\OpnDblParn}{\lgl{E31B}\relax}
\newcommand{\ClsDblParn}{\lgl{E35B}\relax}

\newcommand{\pparen}[1]{\left\OpnDblParn #1\right\ClsDblParn}

\newcommand{\normd}{\lVert\uld\rVert}
\newcommand{\norm}[1]{\left\lVert#1\right\rVert}
\usepackage[pdfencoding=auto,psdextra]{hyperref}
\usepackage[nameinlink]{cleveref}
\Crefname{theorem}{Theorem}{Theorem}
\Crefname{conjecture}{Conjecture}{Conjectures}
\Crefname{lemma}{Lemma}{Lemmas}
\Crefname{definition}{Definition}{Definitions}
\Crefname{remark}{Remark}{Remarks}
\Crefname{question}{Question}{Questions}
\Crefname{proposition}{Proposition}{Propositions}
\Crefname{corollary}{Corollary}{Corollaries}
\Crefname{equation}{}{}
\Crefname{item}{}{}
\Crefname{example}{Example}{Examples}
\Crefname{proof}{Proof}{Proofs}
\crefname{introthm}{Theorem}{Theorems}
\Crefname{introthm}{Theorem}{Theorems}
\newlist{thmenum}{enumerate}{1}
\setlist[thmenum]{label=\arabic*., ref=\thetheorem~(\arabic*)}
\crefalias{thmenumi}{theorem}
\newlist{propenum}{enumerate}{1}
\setlist[propenum]{label=\arabic*., ref=\theproposition~(\arabic*)}
\crefalias{propenumi}{proposition}
\newlist{lemenum}{enumerate}{1}
\setlist[lemenum]{label=\arabic*., ref=\thelemma~(\arabic*)}
\crefalias{lemenumi}{lemma}

\AddToHook{env/theorem/before}{\par\addvspace{3pt}}
\AddToHook{env/theorem/after}{\par\addvspace{3pt}}
\AddToHook{env/example/before}{\par\addvspace{3pt}}
\AddToHook{env/example/after}{\par\addvspace{3pt}}
\AddToHook{env/definition/before}{\par\addvspace{3pt}}
\AddToHook{env/definition/after}{\par\addvspace{3pt}}
\AddToHook{env/proposition/before}{\par\addvspace{3pt}}
\AddToHook{env/proposition/after}{\par\addvspace{3pt}}
\AddToHook{env/lemma/before}{\par\addvspace{3pt}}
\AddToHook{env/lemma/after}{\par\addvspace{3pt}}
\AddToHook{env/remark/before}{\par\addvspace{3pt}}
\AddToHook{env/remark/after}{\par\addvspace{3pt}}
\AddToHook{env/corollary/before}{\par\addvspace{3pt}}
\AddToHook{env/corollary/after}{\par\addvspace{3pt}}

\usepackage[url=true,backend=biber,style=ext-alphabetic,hyperref=true,giveninits=true]{biblatex}
\makeatletter
\renewcommand*\subjclass[2][2020]{%
  \def\@subjclass{#2}%
  \@ifundefined{subjclassname@#1}{%
    \ClassWarning{\@classname}{Unknown edition (#1) of Mathematics
      Subject Classification; using '2020'.}%
  }{%
    \@xp\let\@xp\subjclassname\csname subjclassname@#1\endcsname
  }%
}
\@namedef{subjclassname@1991}{%
  \textup{1991} Mathematics Subject Classification}
\@namedef{subjclassname@2000}{%
  \textup{2000} Mathematics Subject Classification}
\@namedef{subjclassname@2010}{%
  \textup{2010} Mathematics Subject Classification}
\@namedef{subjclassname@2020}{%
  \textup{2020} Mathematics Subject Classification}
\@xp\let\@xp\subjclassname\csname subjclassname@2020\endcsname
\makeatother

\newtheorem{theorem}{Theorem}[section]
\newtheorem{example}[theorem]{Example}
\AddToHook{env/example/begin}{\crefalias{theorem}{example}}
\newtheorem{lemma}[theorem]{Lemma}
\AddToHook{env/lemma/begin}{\crefalias{theorem}{lemma}}
\newtheorem{remark}[theorem]{Remark}
\AddToHook{env/remark/begin}{\crefalias{theorem}{remark}}
\newtheorem{proposition}[theorem]{Proposition}
\AddToHook{env/proposition/begin}{\crefalias{theorem}{proposition}}
\newtheorem{definition}[theorem]{Definition}
\AddToHook{env/definition/begin}{\crefalias{theorem}{definition}}

\AddToHook{env/conjecture/begin}{\crefalias{theorem}{conjecture}}
\newtheorem{corollary}[theorem]{Corollary}
\AddToHook{env/corollary/begin}{\crefalias{theorem}{corollary}}

\AddToHook{env/question/begin}{\crefalias{theorem}{question}}
\newtheorem{assumption}[theorem]{Assumption}
\AddToHook{env/assumption/begin}{\crefalias{theorem}{assumption}}

\newtheorem{introthm}{Theorem}

\Crefformat{enumi}{#2\textup{(#1)}#3}
\usepackage[verbose=silent]{microtype}
\usepackage{geometry}

\numberwithin{equation}{section}

\usepackage{orcidlink}

\title{The order types of the supports of $p$-adic algebraic Hahn series}

\usepackage{orcidlink}
\author{Shanwen Wang\orcidlink{0000-0003-0228-1208}}
\address{School of Mathematics, Renmin University of China, No. 59 Zhongguancun Street, Haidian District, Beijing, 100872, China}
\address{Bigdata and Responsible Artificial Intelligence for National Governance, Renmin University of China, No. 59 Zhongguancun Street, Haidian District, Beijing, 100872, China}
\email{s\_wang@ruc.edu.cn}
\urladdr{https://shanwenwang.github.io/}
\author{Yijun Yuan\orcidlink{0000-0001-6571-6980}}
\address{Institute for Theoretical Sciences, Westlake University, No. 600 Dunyu Road, Sandun town, Xihu district, Hangzhou, Zhejiang Province, 310030, China}
\email{941201yuan@gmail.com}
\urladdr{https://yijunyuan.github.io/}
\begin{abstract} 
Let $p$ be a prime number. In 1986, David Lampert asked if one expands a $\bfQ_p$-algebraic number as a $p$-adic Hahn series, then what is the possible order type of its support? In 2001, Kedlaya made a vague conjecture that the possible ordinals can only be finite ordinals, $\omega$ and $\omega^\omega$, where $\omega$ is the first infinite ordinal. In this article, we solve Lampert's problem by confirming Kedlaya's intuition. Moreover, we completely characterize the $\bfQ_p$-algebraic numbers with order type of support at most $\omega$. As an application, we determine the bi-algebraic set of the $p$-adic shadow map.

\end{abstract}

\usepackage{letltxmacro}
\LetLtxMacro\oldttfamily\ttfamily
\DeclareRobustCommand{\ttfamily}{\oldttfamily\csname ttsize\endcsname}

\begin{document}\def\ttsize{\footnotesize}
\subjclass{11J61, 11J81, 11D88, 68V20}
\maketitle
\tableofcontents

\section{Introduction}

\subsection{Notations and conventions}

Let $p$ be a prime number. Let $\bfF_p$ be the finite field with $p$ elements, and $\overline{\bfF}_p$ be an algebraic closure of $\bfF_p$. Let $\bfQ_p$ be the field of $p$-adic numbers, $\overline{\bfQ}_p$ be an algebraic closure of $\bfQ_p$, $\breve{\bfZ}_p=W(\overline{\bfF}_p)$ be the ring of Witt vectors of $\overline{\bfF}_p$, and $\breve{\bfQ}_p=\breve{\bfZ}_p[p^{-1}]$ which is the completed maximal unramified extension of $\bfQ_p$ with ring of integers $\breve{\bfZ}_p$, ${\bf Q}_p^{\rm un}$ be the maximal unramified extension of ${\bf Q}_p$ and $\bfC_p=\widehat{\overline{\bfQ}}_p$ be the field of $p$-adic complex numbers. We normalize the $p$-adic valuation on $\bfC_p$ and its subfields by $v_p(p)=1$.

Let $\bfL_p\coloneqq\breve{\bfZ}_p\pparen{p^\bfQ}$ (resp. $\bfL_p^\flat\coloneqq\overline{\bfF}_p\pparen{t^\bfQ}$ ) be the field of \textbf{$p$-adic Hahn series} (resp.\textbf{equal-characteristic Hahn series}),
whose elements are the formal sums $f=\sum_{q\in\bfQ}[r_q]p^q$ (resp. $f=\sum_{q\in\bfQ}r_qt^q$) with $r_q\in\overline{\bfF}_p$, $[\cdot]\colon \overline{\bfF}_p\lto \breve{\bfQ}_p$ the Teichm\"uller lift, and well-ordered support $\opn{Supp}(f)=\{q\in\bfQ\colon r_q\neq0\}$; we refer to \Cref{sec:hahn-series} for the precise definitions. For any $f=\sum_{q\in\bfQ}[r_q]p^q\in\bfL_p$ (resp. $f=\sum_{q\in\bfQ}r_q t^q\in\bfL_p^\flat$), we define the coefficient function
$$\opn{Coeff}(f,\cdot)\colon \bfQ\lto\overline{\bfF}_p,\ q\mapsto r_q.$$ By a theorem of Poonen \cite{poonenMAXIMALLYCOMPLETEFIELDS1993}, the field $\bfC_p$ of $p$-adic complex numbers embeds isometrically into $\bfL_p$. We fix such an embedding, and for every rational number $q=a/b\in\bfQ$ with $(a,b)\in\bfZ\times\bfZ_{>0}$, we will use the notation $p^q$ to denote the unique root of the polynomial $X^b-p^a\in\bfQ_p[X]$ with image $p^q\in\bfL_p$ under this embedding. Besides that, for any $f\in\bfL_p$ and any rational number $R\in\bfQ$, write
    $$f_{\leq R}\coloneqq\sum_{q\in\bfQ\cap (-\infty,R]} [r_q]p^q$$
    to be the truncation of $f$ at $R$. 
    
    For a well-ordered subset $S$ of $\bfQ$, let $\operatorname{ot}(S)$ denote
its order type for the usual order on $\bfQ$. We write $\omega$ for the
first infinite ordinal; powers of $\omega$ are ordinal powers.

\subsection{Order types of the supports of $\bfQ_p$-algebraic Hahn series}
Since every $p$-adic algebraic number $f\in\overline{\bfQ}_p$ has a unique expansion as a $p$-adic Hahn series, it is natural to ask how the algebraicity of $f$ is reflected in the shape of its support. In 1986, Lampert (cf. \cite{lampertAlgebraicPadicExpansions1986}) asked which ordinals occur as the order type $\opn{ot}(\opn{Supp}(f))$ of the support of a $\bfQ_p$-algebraic element $f$. Kedlaya proved that this order type is at most $\omega^\omega$ (cf. \cite[Section 4]{kedlayaPowerSeriesPAdic2001}), and expressed the expectation that the only possible values are the finite ordinals, $\omega$ and $\omega^\omega$. The main result of this article confirms this expectation.

\begin{introthm}\label{thm:main}
    Let $f\in\bfL_p$ be algebraic over $\bfQ_p$. Then the order type of the support of $f$ is either finite, $\omega$, or $\omega^\omega$.
\end{introthm}

\Cref{thm:main} is a genuinely mixed-characteristic phenomenon. In the equal-characteristic field of Hahn series $\overline{\bfF}_p\pparen{t^\bfQ}$, Kedlaya \cite{kedlayaAlgebraicClosurePower2001} gave a complete description of the Hahn series which are algebraic over $\overline{\bfF}_p\pparen{t}$ (cf. \Cref{prop:54845}), and the intermediate order types do occur: for $p$ odd, the Hahn series $x=\sum_{k\geq1}t^{-1/p^k}$ is a solution of 
the Artin-Schreier equation $x^p-x=t^{-1}$, and the support of
the square $x^2$ has support of order type $\omega^2$: $\opn{Supp}(x^2)=\{-p^{-i}-p^{-j}\colon 1\leq i\leq j\}$.

The first consequence of \Cref{thm:main} is the following result, which was conjectured in \cite{wang2026padichahnseriessparse} as a weak version of \Cref{thm:main}.

\begin{introthm}[cf. {\Cref{prop:19985}}]\label{thm:20524}
    If $f\in\bfL_p$ is algebraic over $\bfQ_p$ and has bounded support, then its support is finite.
\end{introthm}
Secondly, Krasner's lemma allows us to identify the algebraic elements whose support has order type at most $\omega$.%

\begin{introthm}[cf. {\Cref{prop:57854}}]\label{thm:intro-omega}
    Let $f\in\bfL_p$ be algebraic over $\bfQ_p$. Then $\opn{ot}(\opn{Supp}(f))\leq\omega$ if and only if $f\in\bigcup_{T=1}^\infty\bfQ_p^{\opn{un}}\left(p^{1/T}\right)$.
\end{introthm}

\subsection{Application}

Our application of \Cref{thm:main}  concerns the \textbf{shadow map}
$$\Theta\colon\bfL_p^\flat\lto\bfL_p,\ \sum_{q\in\bfQ}a_qt^q\longmapsto\sum_{q\in\bfQ}[a_q]p^q.$$
It is a set-theoretic bijection which maps $\bfF_p\pparen{t}$ onto $\bfQ_p$. A variant (cf. \Cref{prop:25416}) of \cite[Theorem 13.4]{kedlayaAlgebraicityGeneralizedPower2017b} says that $\Theta$ induces a homeomorphism between the completed algebraic closure of $\bfF_p\pparen{t}$ in $\bfL_p^\flat$ and $\bfC_p$. Without completions, however, $\Theta$ does not preserve algebraicity: in \cite{wang2025padictranscendencesumk1inftyp1pk} we showed that $\sum_{k\geq1}p^{-1/p^k}$ is transcendental over $\bfQ_p$, although its preimage $\sum_{k\geq1}t^{-1/p^k}$ is algebraic over $\bfF_p\pparen{t}$. Combining \Cref{thm:main} and \Cref{thm:intro-omega} with a result of Lisinski (cf. \cite[Theorem 4.5]{lisinskiApproximationAlgebraicityPositive2023}), we determine exactly where $\Theta$ preserves algebraicity.

\begin{introthm}[cf. {\Cref{prop:bialgebraic}}]\label{thm:intro-bialg}
    Let $f\in\bfL_p^\flat$ be a Hahn series. Then the following are equivalent:
    \begin{enumerate}
        \item $f$ is algebraic over $\bfF_p\pparen{t}$ and $\Theta(f)$ is algebraic over $\bfQ_p$;
        \item $\Theta(f)$ belongs to the field $\bigcup_{T=1}^\infty\bfQ_p^{\opn{un}}\left(p^{1/T}\right)$;
        \item $\Theta(f)$ is algebraic over $\bfQ_p$ and $\opn{ot}(\opn{Supp}(\Theta(f)))\leq \omega$.
    \end{enumerate}
\end{introthm}

\subsection{Organization of the paper}
In \Cref{sec:prelim}, we recall the fields of Hahn series, Kedlaya's algebraicity criterion in terms of QTR functions, Kedlaya's theorem on the shadow map and provide the first reduction of the proof of our main theorem by introducing the notion of essential rank and essential truncation point. In \Cref{sec:bounded-identity}, for a $\bfQ_p$-algebraic $p$-adic algebraic Hahn series $f$ with minimal polynomial $P(X)\in \bfZ_p[X]$, we provide a uniform bound for the $p$-adic valuation of $P(f_{\leq B})$ with $B$ greater than the essential truncation point (cf. \Cref{prop:uniform-polynomial-bound}). In \Cref{sec:main}, we deduce \Cref{thm:main} together with its refinement and its application, namely \Cref{thm:20524}(\Cref{prop:19985}), \Cref{thm:intro-omega} (\Cref{prop:57854}) and \Cref{thm:intro-bialg} (\Cref{prop:bialgebraic}).

\subsection{Lean 4 formalization}
All the main results of this paper (cf. \Cref{thm:main}, \Cref{thm:20524}, \Cref{thm:intro-omega} and \Cref{thm:intro-bialg}) have been formalized in Lean 4 (cf. \cite{mouraLean4TheoremProver2021}, \href{https://github.com/leanprover/lean4/releases/tag/v4.33.0}{version 4.33.0}) using the mathlib4 library (cf. \cite{mathlib2020}, \href{https://github.com/leanprover-community/mathlib4/commit/db584cd6d46c92f209a44c0f1c829460d327499d}{commit \texttt{db584cd}}). The code is available at \url{https://github.com/YijunYuan/PAdicOrderType/}.

The formalization is \texttt{sorry}-free and depends only on the standard axioms \texttt{Classical.choice}, \texttt{propext} and \texttt{Quot.sound} of Lean. Moreover, it passes the \href{https://github.com/leanprover/comparator}{Lean Comparator} challenge.
\begin{remark}
    This paper relies on three results (cf. \Cref{prop:54845}, \Cref{thm:47} and \Cref{prop:external-bounds}) taken from Kedlaya's papers \cite{kedlayaAlgebraicClosurePower2001}, \cite{kedlayaPowerSeriesPAdic2001} and \cite{kedlayaAlgebraicityGeneralizedPower2017b}. To keep the formalization entirely \texttt{sorry}-free, we have formalized these three results in Lean 4 as a standalone project, \texttt{TrustworthyKedlaya}, available at \url{https://github.com/YijunYuan/TrustworthyKedlaya}; the formalization of this paper uses \texttt{TrustworthyKedlaya} as a library.

    In addition, we would like to emphasize the following aspects of the formalization:
    \begin{enumerate}
        \item The formal definitions of $\bfL_p$, $\breve{\bfQ}_p$ and $\bfQ_p^{\opn{un}}$ are contained in \texttt{TrustworthyKedlaya}.
        \item The statements of the results from Kedlaya's papers are specialized to the cases of $\bfL_p$ and $\bfL_p^\flat$.
    \end{enumerate}
\end{remark}
\begin{remark}
    We also reuse a few results from our earlier work \cite{wang2026padichahnseriessparse}. To make the code more self-contained, we have copied and merged part of the formalization of \cite{wang2026padichahnseriessparse} into this project; the original formalization is available at \url{https://github.com/YijunYuan/FormalizedSparse/tree/4.33.0}.
\end{remark}

The formalization was carried out by \href{https://github.com/frenzymath/Archon-Horizon}{\texttt{Archon-Horizon}}, an agentic auto-formalization system, together with \href{https://github.com/frenzymath/SHEAF}{\texttt{SHEAF}} (Scalable Hierarchical Engine for Autonomous Formalization), a fully automated formalization framework for Lean 4. Both systems are developed by the AI4Math team at BICMR, Peking University. As base models for these systems, we used \texttt{GPT-6 Astra} (cf. \cite{GPT6Astra}) from OpenAI, as well as \texttt{Claude Fable 5} (cf. \cite{Fable5}) and \texttt{Claude Fable 5.1} (cf. \cite{Fable5-1}) from Anthropic.

\subsection*{Declaration of AI usage}
We use LLMs to polish the English writing of this paper, to check the mathematical statements and proofs, and to assist in the Lean 4 formalization. 

\subsection*{Acknowledgement}
The research is supported by the National Key R\&D Program of China (Grant No. 2024YFA1014000).

\section{Preliminaries}\label{sec:prelim}
\subsection{Hahn series and $p$-adic Hahn series}\label{sec:hahn-series}
Let $R$ be a commutative ring and $G$ be an ordered group. For any $f\in\opn{Hom}_{\opn{Set}}(G,R)$, we define the \textbf{support} of $f$ to be $$\opn{Supp}(f)=\{g\in G\colon f(g)\neq 0\}.$$
The set of \textbf{Hahn series} over $R$ with value group $G$ is defined to be
              $$R\pparen{G}\coloneqq \{f\in \opn{Hom}_{\opn{Set}}(G,R)\colon \opn{Supp}(f) \text{ is well-ordered}\}.$$
              By introducing a formal variable $t$, elements of $R\pparen{G}$ will also be written as $\sum_{g\in G}r_gt^g$, where $r_g\in R$ for all $g\in G$.
 With identity $1\cdot t^0$, and with addition and multiplication given by
              $$\sum_{g\in G}a_g t^g+\sum_{g\in G}b_g t^g\coloneqq\sum_{g\in G}(a_g+b_g)t^g,\ \sum_{g\in G}a_g t^g\cdot\sum_{g\in G}b_g t^g\coloneqq\sum_{g\in G}\left(\sum_{h\in G}a_h b_{g-h}\right)t^g,$$
              $R\pparen{G}$ forms a commutative ring.
If $R$ is a field, then so is $R\pparen{G}$. Moreover, with the map
              $$v\colon R\pparen{G}\lto G\cup\{\infty\},\ f\longmapsto \begin{cases}\min\opn{Supp}(f),& \text{ if }f\neq 0\\\infty,& \text{ if }f=0\end{cases},$$
              $R\pparen{G}$ becomes a valued field with value group $G$ and residue field $R$. Since $\opn{char}R\pparen{G}=\opn{char}R$, we call $R\pparen{G}$ the \textbf{equal-characteristic field of Hahn series} over $R$ with value group $G$, also denoted by $R\pparen{t^G}$ with respect to the formal variable $t$.

Let $k$ be a perfect field of characteristic $p$ and $G$ be an ordered group containing $\bfZ$ as a subgroup. In addition, let
    $$\calN\coloneqq \left\{\sum_{g\in G}r_g t^g\in W(k)\pparen{t^G}\colon \text{ for every } g\in G,\ \sum_{n\in\bfZ}r_{g+n}p^n=0\right\},$$
    where $W(k)$ is the ring of Witt vectors of $k$. We call the elements of $\calN$ the \textbf{null series} of $W(k)\pparen{t^G}$. Then $\calN$ is a maximal ideal of $W(k)\pparen{t^G}$, which makes $W(k)\pparen{p^G}\coloneqq W(k)\pparen{t^G}/\calN$ a field\footnote{Informally, $W(k)\pparen{p^G}$ is obtained by replacing the formal variable $t$ in elements of $W(k)\pparen{t^G}$ by the prime $p$.}, called the \textbf{field of $p$-adic Hahn series}. Every element $x$ in $W(k)\pparen{p^G}$ can be uniquely written as $$x=\sum_{g\in G}[r_g]p^g,$$
              where $r_g\in k$ for all $g\in G$ and $[\cdot]\colon k\lto W(k)$ is the Teichmüller lift. We call this the \textbf{standard expansion} of the element $x$.

        For $f=\sum_{g\in G}[r_g]p^g\in W(k)\pparen{p^G}$, define the \textbf{support} of $f$ to be
              $\opn{Supp}(f)=\{g\in G\colon r_g\neq 0\}$.
              Then the map
              $$v\colon W(k)\pparen{t^G}/\calN\lto G\cup\{\infty\},\ f\mapsto \begin{cases}
                      \min\opn{Supp}(f), & \text{ if }f\neq 0 \\\infty,& \text{ if }f=0
                  \end{cases}$$
              makes $W(k)\pparen{t^G}/\calN$ a mixed-characteristic valued field with value group $G$ and residue field $k$.
   
The most fundamental property of the field of Hahn series is the following:
\begin{theorem}[cf. {\cite[Theorem 1, Corollary 4, Corollary 6]{poonenMAXIMALLYCOMPLETEFIELDS1993}}]\label{prop:49318}
    Let $F$ be an equal-characteristic (resp. mixed-characteristic) valued field with divisible value group $G$ and algebraically closed residue field $k$. Then the equal-characteristic (resp. $p$-adic) field of Hahn series $k\pparen{t^G}$ (resp. $W(k)\pparen{p^G}$) is the unique (up to isomorphism of valued fields) minimal spherically complete extension of $F$. Moreover, it is algebraically closed and complete.
\end{theorem}

The following are the fields of Hahn series used in this article:
\begin{example}\label{eg:44541}
    Let $F$ be the algebraic closure of $\overline{\bfF}_p\pparen{t}$ (resp. $\breve{\bfQ}_p$), which has value group $\bfQ$ and residue field $\overline{\bfF}_p$. Then the field of equal-characteristic (resp. $p$-adic) Hahn series $\bfL_p^\flat\coloneqq\overline{\bfF}_p\pparen{t^\bfQ}$ (resp. $\bfL_p\coloneqq\breve{\bfZ}_p\pparen{p^\bfQ}$) is the spherical completion of $F$ with the same residue field and value group, and is algebraically closed and complete.
\end{example}
In the rest of the paper, we refer to elements of $\bfL_p$ as $p$-adic Hahn series, without specifying the residue field and the value group, for brevity.

\subsection{QTR functions and $\overline{\bfF}_p\pparen{t}$-algebraicity criterion}
In \cite{kedlayaAlgebraicClosurePower2001}, Kedlaya gave a necessary and sufficient condition for a Hahn series to be algebraic over $\overline{\bfF}_p\pparen{t}$. In this subsection, we recall the definition of QTR functions and Kedlaya's algebraicity criterion in terms of the notion of words and digit vectors that are introduced in \cite{wang2026padichahnseriessparse}.

\begin{definition}\leavevmode
    \begin{enumerate}
        \item Let $\bbP\coloneqq \{\uld\in\bigoplus_{\bfZ_{\geq1}}\bfN\mid 0\leq \uld_i<p\text{ for every }i\in\bfZ_{>0}\}\subsetneq \bigoplus_{\bfZ_{\geq 1}}\bfN$. Elements of $\bbP$ are called \textbf{digit vectors}.
        \item For $\uld\in \bigoplus_{\bfZ_{\geq 1}}\bfN$ and $i\in\bfZ_{\geq 1}$, denote by $d_i=\uld_i$ the $i$-th component of $\uld$. The direct sum consists of maps $\bfZ_{>0}\to\bfN$ that vanish outside a finite set, with coordinatewise addition.
        \item For $\uld\in \bigoplus_{\bfZ_{\geq 1}}\bfN$, call $\Psi(\uld)\coloneqq \sum_{i=1}^\infty \uld_i\in\bfN$ the \textbf{digit sum} of $\uld$, $\normd\coloneqq \sum_{i=1}^\infty \uld_i p^{-i}\in\bfQ_{\geq 0}$ the \textbf{digit value} of $\uld$, and $\opn{supp}(\uld)=\{i\in\bfZ_{\geq 1}\mid \uld_i\neq 0\}$ the \textbf{support} of $\uld$.
    \end{enumerate}
\end{definition}
\begin{lemma}[{\cite[Section 3]{wang2026padichahnseriessparse}}]\label{lem:20585}
    For any $\uld\in\bigoplus_{\bfZ_{\geq 1}}\bfN$, there exists a unique element $\tau(\uld)\in\bbP$ such that $\normd-\lVert\tau(\uld)\rVert\in\bfZ$.
\end{lemma}
We call $\tau(\uld)$ the \textbf{normalization} of $\uld$.
More generally, let $T\geq1$ be an integer and $x\in T^{-1}\bfZ[1/p]$. The fractional part of $-Tx$ lies in $\bfZ[1/p]\cap[0,1)$ and thus has a unique terminating base-$p$ expansion; in other words, there is a unique $\uld\in\bbP$ such that $Tx+\normd\in\bfZ$. We call it the digit vector \textbf{$T$-representing} $x$. Concretely, $\normd$ is the fractional part of $-Tx$, so that $x=(\lceil Tx\rceil-\normd)/T$. In particular, $\tau(\uld)$ is the digit vector $1$-representing $-\normd$: it records the fractional digits after carrying in a coordinatewise sum.

For integers $a\geq 1$ and $b,c\geq 0$, we set
$$S_{a,b,c} = \left\{\frac{1}{a}\left(m-\norm{\uld}\right)\middle\vert m\in\bfZ,\ m\geq -b,\ \uld\in\bbP,\ \Psi(\uld)\leq c \right\}\subset \frac{1}{a}\bfZ[1/p].$$
For $m\geq -b$, we define the $m$-th slice of $S_{a,b,c}$ by
$$S_{a,b,c,m}=\left\{\frac{1}{a}\left(m-\norm{\uld}\right)\middle\vert\uld\in\bbP,\ \Psi(\uld)\leq c\right\}\subset \frac{1}{a}\bfZ[1/p]\cap\left(\frac{m-1}{a},\frac{m}{a}\right].$$
\begin{definition}    
We say that a function $x\colon\bfQ\lto\overline{\bfF}_p$ is quasi-twist recurrent (QTR for short) with respect to the data $(a,b,c,M,N)\in \bfZ_{>0}\times \bfN\times \bfN\times \bfZ_{>0}\times \bfZ_{>0}$, if
\begin{enumerate}
  \item $\opn{Supp}(x)\coloneqq\{q\in \bfQ\vert x(q)\neq 0\}$ is well-ordered and contained in $S_{a,b,c}$;
  \item whenever $\uld\in\bbP$, $\Psi(\uld)\leq c$, and                      $d_{s+1}=\cdots=d_{s+M}=0$, for some integer $s\geq 0$, the insertion of $N$ further zeros immediately after this run yields a digit vector $\uld' $ satisfying
  $$x\left(\frac{1}{a}(m-\norm{\uld})\right)=x\left(\frac{1}{a}(m-\norm{\uld'})\right)$$ for every integer $m\geq -b$.
  \end{enumerate}
We omit the data $(a,b,c,M,N)$ when they are clear from the context, and we say that a subset of $\bfQ$ is QTR if it is the support of a QTR function.
\end{definition} 

\begin{theorem}[cf. {\cite[Theorem 15]{kedlayaAlgebraicClosurePower2001}, \cite[Theorem 11.11]{kedlayaAlgebraicityGeneralizedPower2017b}}]\label{prop:54845}
    A Hahn series $x\in\bfL_p^\flat$ is algebraic over $\overline{\bfF}_p\pparen{t}$ if and only if the function $\opn{Coeff}(x,\cdot)$ is QTR.
\end{theorem}

\subsection{Shadow map and completed $p$-adic algebraic closure}
\begin{definition}\label{def:shadow-map}
    Define the \textbf{shadow map} to be the set-theoretic bijection
    $$\Theta\colon \bfL_p^\flat\lto\bfL_p,\ \sum_{q\in\bfQ}a_qt^q\longmapsto\sum_{q\in\bfQ}[a_q]p^q.$$
\end{definition}

The following result of Kedlaya links the algebraicity of $p$-adic Hahn series to that of equal-characteristic Hahn series:
\begin{theorem}[{cf. \cite[Theorem 13.4]{kedlayaAlgebraicityGeneralizedPower2017b}}]\label{thm:47}
   Let $L$ be the completed algebraic closure of $\overline{\bfF}_p\pparen{t}$ in $\overline{\bfF}_p\pparen{t^\bfQ}$. Then the completed integral closure of $\breve{\bfQ}_p$ in $\bfL_p$ is the completion of the following set:
    $$\left\{\sum_q[x_q]p^q\in\bfL_p\middle\vert\sum_q x_q t^q\in L\right\}.$$
\end{theorem}

This result can be reformulated as follows:
\begin{proposition}\label{prop:25416}
    The shadow map $\Theta$ gives a homeomorphism between the completed algebraic closure of $\bfF_p\pparen{t}$ and $\bfC_p$.
\end{proposition}
\begin{proof}
    Since $\bfC_p$ is algebraically closed, the completed algebraic closure of $\breve{\bfQ}_p$ identifies with $\bfC_p$. Let $\bfK$ (resp. $\bfK_0$) be the algebraic closure of $\overline{\bfF}_p\pparen{t}$ (resp. $\bfF_p\pparen{t}$) in $\overline{\bfF}_p\pparen{t^\bfQ}$. We claim that $\widehat{\bfK}$ and $\widehat{\bfK}_0$ coincide, where the completion is taken with respect to the $t$-adic valuation. The inclusion $\widehat{\bfK}_0\subseteq \widehat{\bfK}$ is clear. Conversely, the continuity of roots (cf. \cite[Chapter 3, Section 3.4.1]{boschNonArchimedeanAnalysisSystematic1984a}) reduces the claim to showing that $\overline{\bfF}_p\pparen{t}$ lies in $\widehat{\bfK}_0$. This follows from the fact that $\overline{\bfF}_p(t)$ is a dense subfield of $\overline{\bfF}_p\pparen{t}$, and is algebraic over $\bfF_p\pparen{t}$.

    The result follows from \Cref{thm:47} and the fact that $v_t(x-y)=v_p(\Theta(x)-\Theta(y))$ for every $x,y\in\bfL_p^\flat$.
    
\end{proof}

Note that, using \Cref{prop:25416}, we can transfer the $\overline{\bfF}_p\pparen{t}$-algebraicity criterion (cf. \Cref{prop:54845}) to a necessary condition for a $p$-adic Hahn series to be algebraic over $\bfQ_p$:
\begin{corollary}\label{lem:local-qtr-full}
Let $f\in \bfL_p$ be algebraic over $\bfQ_p$. For every $R\in \bfZ$, the function $\opn{Coeff}(f_{\leq R},\cdot)$
is QTR. Its data $(a,b,c,M,N)$ may depend on $f$ and $R$.
\end{corollary}
\begin{proof}By \Cref{prop:25416}, we know that there exists a $\bfF_p\pparen{t}$-algebraic Hahn series $g\in \bfL_p^\flat$ such that $\opn{Coeff}(g,q)=\opn{Coeff}(f,q)=\opn{Coeff}(f_{\leq R},q)$ for all $q\leq R$. By \Cref{prop:54845}, the function $\opn{Coeff}(g,\cdot)$ is QTR, and the result follows from \cite[Lemma 6.7]{wang2026padichahnseriessparse}.
\end{proof}
Recall that the following result of Kedlaya provides an upper bound for the order type of the support of a $\bfQ_p$-algebraic $p$-adic Hahn series:
\begin{proposition}[{\cite[Section 4]{kedlayaPowerSeriesPAdic2001}}]
\label{prop:external-bounds}
If $f\in\bfL_p$ is algebraic over $\bfQ_p$, then
$\operatorname{ot}(\operatorname{Supp}(f))\leq\omega^\omega$.
\end{proposition}

We finish this section by giving a reduction of the proof of \Cref{thm:main} by introducing the notion of essential rank and essential truncation point of a $p$-adic algebraic Hahn series. The following lemma relates the order type of the support to those of
its bounded truncations.
\begin{lemma}\label{lem:14033}
For $f\in\bfL_p\setminus\{0\}$ with
$\operatorname{ot}(\operatorname{Supp}(f))<\omega^\omega$, put
\[
 W_f=\{r\in\bfN:\text{there exists }B\in\bfZ\text{ with }
       \operatorname{ot}(\operatorname{Supp}(f_{\leq B}))\geq\omega^r\}.
\]
Then:
\begin{enumerate}
\item $\bfr(f)\coloneqq\max W_f$ is a nonnegative integer, called the
\textbf{essential rank} of $f$;
\item for every $B\in\bfZ$,
$\operatorname{ot}(\operatorname{Supp}(f_{\leq B}))<\omega^{\bfr(f)+1}$;
\item if $\bfr(f)=0$, then
$\operatorname{ot}(\operatorname{Supp}(f))\leq\omega$.
\end{enumerate}
\end{lemma}
\begin{proof}
Since $f\neq0$, some integer truncation has nonempty support, so
$0\in W_f$. The hypothesis on the order type gives an integer $s\geq1$
such that $\operatorname{ot}(\operatorname{Supp}(f))<\omega^s$.
Hence $W_f\subseteq\{0,\cdots,s-1\}$, proving (1).
Assertion (2) follows from the maximality of $\bfr(f)$.
If $\bfr(f)=0$, (2) says that every integer truncation has finite
support. Thus every point of $\operatorname{Supp}(f)$ has only finitely
many predecessors, and its order type is at most $\omega$.
\end{proof}

\begin{definition}\label{def:essential-truncation}
For $f$ as in \Cref{lem:14033}, the \textbf{essential truncation point}
$\bfB(f)$ is the least integer $B$ such that
\[
 \operatorname{ot}(\operatorname{Supp}(f_{\leq B}))\geq\omega^{\bfr(f)}.
\]
Such integers exist by the definition of $\bfr(f)$ and are bounded
below by $v_p(f)$, so the least one exists.
\end{definition}

By \Cref{lem:14033}, to prove \Cref{thm:main}, it suffices to show that a nonzero $\bfQ_p$-algebraic Hahn series whose support has order type less than $\omega^\omega$ cannot have positive essential rank.

\section{Essential truncation and uniform polynomial estimation}
\label{sec:bounded-identity}

Let $P\in\bfQ_p[X]$ be the monic minimal polynomial of a nonzero $\bfQ_p$-algebraic Hahn series $f$. Since $P(f)=0$, the $p$-adic valuation $v_p(P(f_{\leq B}))$ tends to infinity as $B$ increases. The aim of this section is to prove the following proposition, which gives a uniform upper bound of $v_p(P(f_{\leq B}))$ with $B>\bfB(f)$ when the essential rank of $f$ is positive, leading to the desired contradiction.

\begin{proposition}
\label{prop:uniform-polynomial-bound}
For the fixed prime $p$, there exists a function
$C\colon\bfZ_{\geq1}^2\to\bfQ_{\geq0}$ with the following property.
For any $f\in\calO_{\overline{\bfQ}_p}\setminus\{0\}$ with
$\operatorname{ot}(\operatorname{Supp}(f))<\omega^\omega$ and $\bfr(f)>0$,
any monic polynomial $P(X)=X^n+\sum_{i=0}^{n-1}a_iX^i\in\bfZ_p[X]$
with $\deg(P)\geq1$, and an integer $B>\bfB(f)$,
one has
\begin{equation}\label{eq:uniform-polynomial-bound}
 v_p(P(f_{\leq B}))\leq \deg(P)\cdot(\bfB(f)+1)+C(\deg(P),\bfr(f)).
\end{equation}
\end{proposition}
\begin{remark}
    The right-hand side of the inequality (\ref{eq:uniform-polynomial-bound}) is independent of the truncation point $B>\bfB(f)$ and of the QTR data of $f_{\leq B}$. Let us explain the strategy of the proof of this proposition. Put $g=f_{\leq B}$ and $n=\deg(P)$. In \Cref{sec:coefficient-reduction}, we sort the coefficients of $g$ according to the digit vectors of their exponents and show that a suitable coefficient of the \textbf{carry-free part} of $g^n$ bounds $v_p(P(g))$ from above (cf. \Cref{lem:truncation-coset}). We construct such a coefficient in the remaining subsections: we translate the order-type hypotheses of the support of $g$ into statements about digit vectors in \Cref{sec:digit-blocks}, identify the coefficient with a sum indexed by interleavings of words in \Cref{sec:block-coefficients}, and bound the cancellation in such a sum by a constant depending only on $n$ and $\bfr(f)$ in \Cref{sec:bounded-coefficient}.
\end{remark}

\subsection{Reduction to a carry-free coefficient}
\label{sec:coefficient-reduction}

Throughout this subsection, we fix integers $R\geq0$ and $T\geq1$, and an element $g\in\bfL_p$ whose support $V\coloneqq\opn{Supp}(g)$ is contained in $[0,R]\cap T^{-1}\bfZ[1/p]$. For $\uld\in\bbP$, we put
\begin{equation}\label{eq:aggregate-valuation}
 A_{\uld}\coloneqq\sum_{l=0}^{TR}\left[\opn{Coeff}\left(g,\frac{l-\normd}{T}\right)\right]p^{l/T}\in\breve{\bfZ}_{p,T}\coloneqq\breve{\bfZ}_p[p^{1/T}].
\end{equation}
If $q\in V$ is $T$-represented by $\uld$, then $l\coloneqq Tq+\normd$ is an integer with $0\leq l\leq TR$, because $0\leq q\leq R$ and $0\leq\normd<1$; hence every point of $V$, $T$-represented by $\uld$, contributes a term to the sum \eqref{eq:aggregate-valuation}, and all other terms vanish. Since the nonzero terms of the sum \eqref{eq:aggregate-valuation} have pairwise distinct valuations $l/T$, we see that $A_{\uld}\neq0$ if and only if some point of $V$ is $T$-represented by the digit vector $\uld$, in which case
$$v_p(A_{\uld})=\min\left\{q+\frac{\normd}{T}\colon q\in V\text{ is $T$-represented by }\uld\right\}.$$
In particular, if a point of $V\cap[0,R_*]$ for some integer $R_*$ is $T$-represented by $\uld\in\bbP$, then
\begin{equation}\label{eq:small-truncation-coefficient}
 \  v_p(A_{\uld})<R_*+1.
\end{equation}

For an integer $n\geq1$ and $\ulu\in\bbP$, we set
\begin{equation}\label{eq:carry-free-coefficient}
 \calS_n(\ulu)\coloneqq\sum_{\substack{\uld_1,\cdots,\uld_n\in\bbP\\ \uld_1+\cdots+\uld_n=\ulu}}A_{\uld_1}\cdots A_{\uld_n},
\end{equation}
where the coordinatewise sum is taken over the $n$-tuples of digit vectors. This is a finite sum, as each $\uld_i$ is bounded coordinatewise by $\ulu$. Informally, $\calS_n(\ulu)$ collects those products of $n$ coefficients of $g$ for which the digit vectors of the exponents add up to $\ulu$ without any carry. 

\begin{definition}\label{def:grtn}
We say that $\ulu\in\bbP$ satisfies the $(g,R,T,n)$-\textbf{rigid} condition, if for every integer $0\leq k\leq n$ and all $\uld_1,\cdots,\uld_k\in\bbP$ with $A_{\uld_1},\cdots,A_{\uld_k}\neq0$,
\begin{equation}\label{eq:collapse-condition}
 \tau(\uld_1+\cdots+\uld_k)=\ulu\quad\Longrightarrow\quad k=n\ \text{ and }\ \uld_1+\cdots+\uld_k=\ulu,
\end{equation}
where for $k=0$ the empty sum is the zero vector.
\end{definition}

\begin{proposition}\label{lem:truncation-coset}
Let $g,T,R$ and $(A_{\uld})_{\uld\in\bbP}$ be as above, and let $$P(X)=X^n+\sum_{k=0}^{n-1}a_k X^k\in\bfZ_p[X]$$ be a monic polynomial of degree $n\geq1$. If $\ulu\in\bbP$ satisfies the $(g,R,T,n)$-rigid condition, then $v_p(P(g))\leq v_p(\calS_n(\ulu))$.
\end{proposition}

\begin{proof}
We first prove the identity
\begin{equation}\label{eq:coset-collapse}
 \sum_{l\in\bfZ}\left[\opn{Coeff}\left(P(g),\frac{l-\norm{\ulu}}{T}\right)\right]p^{l/T}=\calS_n(\ulu).
\end{equation}
Let us first make sense of the left-hand side. Since $g$ and $P$ have integral coefficients, so does $P(g)$; thus $\opn{Supp}(P(g))$ is a well-ordered subset of $[0,\infty)$, and it meets the set $-\norm{\ulu}/T+T^{-1}\bfZ$ in a well-ordered set of points $(l-\norm{\ulu})/T$ with $l$ bounded below. Therefore, the nonzero terms of the left-hand side $\sum_{l\in\bfZ}\left[\opn{Coeff}\left(P(g),\frac{l-\norm{\ulu}}{T}\right)\right]p^{l/T}$ of the equality \eqref{eq:coset-collapse} have valuations $l/T$ which are bounded below and tend to infinity, so this series converges in the complete field $\breve{\bfQ}_{p,T}\coloneqq\breve{\bfQ}_p(p^{1/T})$. %

Consider the Hahn series $\wtilde g\coloneqq\sum_{q\in V}[\opn{Coeff}(g,q)]t^q$ and $\wtilde h\coloneqq\sum_{q\in\bfQ}[\opn{Coeff}(P(g),q)]t^q$ in $\breve{\bfZ}_p\pparen{t^\bfQ}$. Both $P(\wtilde g)$ and $\wtilde h$ have image $P(g)$ in $\bfL_p=\breve{\bfZ}_p\pparen{t^\bfQ}/\calN$, so their difference is a null series. As a result, if we write $P(\wtilde g)=\sum_{q\in\bfQ}c_qt^q$ with $c_q\in\breve{\bfZ}_p$, then for every $q_0\in\bfQ$, we have
\begin{equation}\label{eq:648653}
\sum_{i\in\bfZ}c_{q_0+i}\,p^{i}=\sum_{i\in\bfZ}\left[\opn{Coeff}(P(g),q_0+i)\right]p^{i}.
\end{equation}
Note that the left-hand side of the above equality is a finite sum, because $\opn{Supp}(P(\wtilde g))\subseteq[0,nR]$, while the right-hand side is a convergent series in $\breve{\bfQ}_p$. 

Now let $0\leq j<T$. Applying the identity \eqref{eq:648653} to $q_0=(j-\norm{\ulu})/T$ and multiplying by $p^{j/T}$, we get
$$\sum_{i\in\bfZ}c_{(j+Ti-\norm{\ulu})/T}\,p^{(j+Ti)/T}=\sum_{i\in\bfZ}\left[\opn{Coeff}\left(P(g),\frac{j+Ti-\norm{\ulu}}{T}\right)\right]p^{(j+Ti)/T}.$$
Summing these $T$ identities over $j$, we obtain: 
\[\sum_{0\leq j< T}\sum_{i\in\bfZ}c_{(j+Ti-\norm{\ulu})/T}\,p^{(j+Ti)/T}=\sum_{0\leq j<T}\sum_{i\in\bfZ}\left[\opn{Coeff}\left(P(g),\frac{j+Ti-\norm{\ulu}}{T}\right)\right]p^{(j+Ti)/T}.\] Since every $l\in\bfZ$ can be written uniquely as $l=j+Ti$ with $0\leq j<T$ and $i\in\bfZ$, the left-hand sides add up to the finite sum $\sum_{l\in\bfZ}c_{(l-\norm{\ulu})/T}\,p^{l/T}$. The right-hand sides for $0\leq j< T$ add up to the left-hand side of the equality \eqref{eq:coset-collapse}: indeed, the terms of that series tend to zero, so it converges unconditionally, and grouping its terms according to the residue of $l$ modulo $T$ does not change its sum. Hence
$$\sum_{l\in\bfZ}\left[\opn{Coeff}\left(P(g),\frac{l-\norm{\ulu}}{T}\right)\right]p^{l/T}=\sum_{l\in\bfZ}c_{(l-\norm{\ulu})/T}\,p^{l/T}.$$

We now compute the coefficients $c_q$ by expanding $P(\wtilde g)=\wtilde g^{\,n}+\sum_{k=0}^{n-1}a_k\wtilde g^{\,k}$. For $0\leq k\leq n$, the coefficient of $t^q$ in $\wtilde g^{\,k}$ is the sum of $\prod_{i=1}^k[\opn{Coeff}(g,q_i)]$ over the finitely many $(q_1,\cdots,q_k)\in V^k$ with $q_1+\cdots+q_k=q$; for $k=0$ the only such tuple is the empty one, with $q=0$. Let $(q_1,\cdots,q_k)$ be such a tuple, and let $\uld_i$ be the digit vector $T$-representing $q_i$. As a result, $A_{\uld_i}\neq0$ and $q_i=(l_i-\norm{\uld_i})/T$ with $0\leq l_i\leq TR$. Since $\norm{\uld_1+\cdots+\uld_k}=\norm{\uld_1}+\cdots+\norm{\uld_k}$, the exponent $q_1+\cdots+q_k$ lies in $-\norm{\ulu}/T+T^{-1}\bfZ$ if and only if $\tau(\uld_1+\cdots+\uld_k)=\ulu$ by \Cref{lem:20585}. Since $\ulu$ satisfies the $(g,R,T,n)$-rigid condition, we have $k=n$ and $\uld_1+\cdots+\uld_n=\ulu$. As a result, only the term $\wtilde g^{\,n}$ contributes to $\sum_{l\in \bfZ}c_{(l-\norm{\ulu})/T}\,p^{l/T}$, and this sum equals
\begin{align*}
&\sum_{\substack{\uld_1,\cdots,\uld_n\in\bbP\\ \uld_1+\cdots+\uld_n=\ulu}}\sum_{0\leq l_1,\cdots,l_n\leq TR}\prod_{i=1}^n\left[\opn{Coeff}\left(g,\frac{l_i-\norm{\uld_i}}{T}\right)\right]p^{l_i/T}\\
=&\sum_{\substack{\uld_1,\cdots,\uld_n\in\bbP\\ \uld_1+\cdots+\uld_n=\ulu}}\prod_{i=1}^n\sum_{l_i=0}^{TR}\left[\opn{Coeff}\left(g,\frac{l_i-\norm{\uld_i}}{T}\right)\right]p^{l_i/T}\\
=&\sum_{\substack{\uld_1,\cdots,\uld_n\in\bbP\\ \uld_1+\cdots+\uld_n=\ulu}}\prod_{i=1}^nA_{\uld_i}=\calS_n(\ulu),
\end{align*}
where we have used that the terms with $(l_i-\norm{\uld_i})/T\notin V$ vanish. This proves the equality \eqref{eq:coset-collapse}.

Finally, if $P(g)=0$, then $\calS_n(\ulu)=0$. Otherwise, a term $\left[\opn{Coeff}\left(P(g),\frac{l-\norm{\ulu}}{T}\right)\right]p^{l/T}$ on the left-hand side of the equality \eqref{eq:coset-collapse} is nonzero only if $(l-\norm{\ulu})/T\in\opn{Supp}(P(g))$. As a result, we have $l/T\geq v_p(P(g))+\norm{\ulu}/T\geq v_p(P(g))$. Hence, by the ultrametric inequality, which passes to convergent series, the left-hand side of the equality \eqref{eq:coset-collapse} has valuation at least $v_p(P(g))$, and so does $\calS_n(\ulu)$.
\end{proof}

By \Cref{lem:truncation-coset}, in order to prove \Cref{prop:uniform-polynomial-bound} it suffices to find, for $g=f_{\leq B}$ and a suitable $T$, a digit vector $\ulu$ satisfying the condition \eqref{eq:collapse-condition} such that $\calS_n(\ulu)$ is nonzero of valuation at most $n(\bfB(f)+1)+C$, with $C$ independent of $B$. 

\subsection{From order type to separated digit blocks}
\label{sec:long-gaps}\label{sec:digit-blocks}
Let $V$ be a bounded QTR set with QTR data $(a,b,c,M,N)$ and of order type $<\omega^{r+1}$ with $r\in \bfN_{>0}$. This subsection is purely combinatorial: we relate the order type of a bounded QTR set $V$ to the digit vectors of certain points in $V$, which are possible candidates required in the previous subsection.

\begin{definition}\label{def:zero-gaps}
Let $\uld\in\bbP$ be a nonzero digit vector with nonzero positions $\xi_1<\cdots<\xi_\rho$, and put $\xi_0=0$. The \textbf{zero gaps} of $\uld$ are the $\rho$ (possibly empty) strings of zeros occupying the positions $\xi_{j-1}+1,\cdots,\xi_j-1$, for $1\leq j\leq\rho$; the $j$-th gap has length $\xi_j-\xi_{j-1}-1$, and it is called \textbf{long} if its length is at least $M$. Trailing zeros are not counted, and the zero vector has no gaps.
\end{definition}

Every $q\in a^{-1}\bfZ[1/p]$ can be written uniquely as $q=(m-\normd)/a$ with $m\in\bfZ$ and $\uld\in\bbP$, namely with $m=\lceil aq\rceil$ and $\uld$ the digit vector $a$-representing $q$. If $q\in S_{a,b,c}$, then $m\geq-b$ and $\Psi(\uld)\leq c$ by the definition of $S_{a,b,c}$. The zero gaps of $q$ are, by definition, those of this $\uld$.

\begin{lemma}\label{lem:long-gap-criterion}
Let $U\subseteq\bfQ$ be a bounded QTR set with data $(a,b,c,M,N)$, and let $k\geq1$ be an integer. Then $\opn{ot}(U)<\omega^k$ if and only if every point of $U$ has fewer than $k$ long gaps. Consequently, if $r\geq1$ is an integer with $\omega^r\leq\opn{ot}(U)<\omega^{r+1}$, then every point of $U$ has at most $r$ long gaps, and some point of $U$ has exactly $r$ long gaps.
\end{lemma}
\begin{proof}
In this proof, tuples of integers of the same length are compared lexicographically.%

Suppose first that some point $q=(m-\normd)/a$ of $U$ has at least $k$ long gaps, and choose $k$ of them. For any $\uln=(n_1,\cdots,n_k)\in\bfN^k$, let $q_{\uln}\in U$ be the point obtained from $q$ by inserting $Nn_j$ zeros into the $j$-th chosen gap for $1\leq j\leq k$, keeping $m$, and let $\uld^{(\uln)}$ be its digit vector. If $\uln<\uln'$ lexicographically, say $n_j<n'_j$ at the first differing coordinate $j$, then $\uld^{(\uln)}$ and $\uld^{(\uln')}$ agree in all positions up to the end of the $j$-th chosen gap of $\uld^{(\uln)}$, and at the following position $\uld^{(\uln)}$ has a nonzero digit while $\uld^{(\uln')}$ still has a zero. As a result, one has $\norm{\uld^{(\uln)}}>\norm{\uld^{(\uln')}}$, i.e. $q_{\uln}<q_{\uln'}$. This shows that the map $\uln\longmapsto q_{\uln}$ is an order embedding of $\bfN^k$, ordered lexicographically, into $U$. Since $\bfN^k$ has order type $\omega^k$, we get $\opn{ot}(U)\geq\omega^k$.

Conversely, suppose that every point of $U$ has fewer than $k$ long gaps. Partition $U$ into classes by attaching to each point $q=(m-\normd)/a$ the following data: the integer $m$, the sequence of nonzero digits of $\uld$ in their order, the set of indices $j$ for which the $j$-th gap is long, and the lengths of the remaining gaps. There are only finitely many classes: $m$ ranges over the integers between $-b$ and $\lceil a\sup U\rceil$, the number of nonzero digits is at most $\Psi(\uld)\leq c$, and the gaps that are not long have length less than $M$. Consider a class with $j<k$ long gaps. A point of this class is determined by the lengths of its long gaps, and subtracting $M$ from each of these lengths identifies the class with a subset of $\bfN^j$. By similar discussion as above, this identification is an order embedding for the lexicographic order, and the class has order type at most $\omega^j<\omega^k$. Since $U$ is a finite union of such classes, one has $\opn{ot}(U)<\omega^k$ by \cite[Proposition 5.3(1)]{altman-order}.

Applying the equivalence with $k=r+1$ and with $k=r$ gives the last assertion.
\end{proof}

\begin{definition}\label{def:digit-patterns}
Let $s\geq1$ be an integer. 
\begin{enumerate}
\item A \textbf{digit pattern} of length $s$ is a tuple $\alpha=(\alpha_1,\cdots,\alpha_s)\in\{0,\cdots,p-1\}^s$; it is \textbf{nonzero} if some $\alpha_i$ is nonzero. 
\item For integers $A\leq B$, we write $[A,B]_{\bfZ}\coloneqq\{k\in\bfZ\colon A\leq k\leq B\}$. A \textbf{block} of length $s$ is a set $[e,e+s-1]_{\bfZ}$ of $s$ consecutive positive integers. 
\item \textbf{Placing} a digit pattern $\alpha$ on the block $[e,e+s-1]_{\bfZ}$ means assigning the digit $\alpha_i$ to the position $e+i-1$ for $1\leq i\leq s$.
\end{enumerate}
\end{definition}

\begin{lemma}\label{lem:tail-clusters}
Let $h\geq cM$ be an integer and $q\in V$. Then the support of the digit vector $ap^h$-representing $q$ is the union of at most $r$ nonempty sets of diameter at most $cM$.
\end{lemma}
\begin{proof}
Write $q=(m-\normd)/a$ as in \Cref{def:zero-gaps}, so that $ap^hq=p^hm-p^h\normd$. Since
$$p^h\normd=\sum_{i=1}^hd_ip^{h-i}+\sum_{i\geq1}d_{h+i}p^{-i},$$
we have $ap^hq+\sum_{i\geq1}d_{h+i}p^{-i}\in\bfZ$, so the digit vector $ap^h$-representing $q$ is $(d_{h+1},d_{h+2},\cdots)$, which is obtained from $\uld$ by deleting the first $h$ digits. If $\uld=0$ there is nothing to prove. Otherwise, let $\xi_1<\cdots<\xi_\rho$ be the nonzero positions of $\uld$; then $\rho\leq\Psi(\uld)\leq c$, and $\uld$ has at most $r$ long gaps by \Cref{lem:long-gap-criterion}. Cut the sequence $\xi_1,\cdots,\xi_\rho$ into consecutive segments immediately after each long gap. This gives at most $r+1$ segments: an initial segment, consisting of the nonzero positions not preceded by any long gap, which may be empty (this happens precisely when the first gap is long); and one nonempty segment beginning right after each long gap. Within a segment, consecutive positions are at distance at most $M$, so each segment has diameter at most $(\rho-1)M<cM$. Moreover, the $k$-th position of the initial segment is at most $kM$, so this segment is contained in $[1,\rho M]_{\bfZ}\subseteq[1,h]_{\bfZ}$ and disappears when the first $h$ digits are deleted. Hence the support of $(d_{h+1},d_{h+2},\cdots)$ is the union of the sets $\{\xi-h\colon\xi\in S,\ \xi>h\}$, where $S$ runs over the at most $r$ remaining segments; discarding the empty ones, these are at most $r$ nonempty sets of diameter at most $cM$.
\end{proof}

The next lemma constructs points whose digit vectors, after a suitable rescaling, consist of fixed nonzero patterns that can be moved along a widely spaced grid of blocks while preserving their order.

\begin{lemma}\label{lem:separated-placements}
Let $R_*$ be an integer such that $U\coloneqq V\cap(-\infty,R_*]$ has order type at least $\omega^r$, and let $n\geq1$ be an integer. Then there exist 
\begin{enumerate}
    \item integers $h\geq cM$, $s\geq1$, $\beta\geq1$;
    \item a positive multiple $\Delta$ of $N$ with
\begin{equation}\label{eq:block-spacing}
 \Delta>(s-1)+n+2cM+M,
\end{equation}
\item nonzero digit patterns $\alpha^{(1)},\cdots,\alpha^{(r)}$ of length $s$,
\end{enumerate}
such that if we set $T\coloneqq ap^h$ and $B_j\coloneqq[\beta+j\Delta,\beta+j\Delta+s-1]_{\bfZ}$ for $j\geq1$, then for all integers $1\leq j_1<\cdots<j_r$, there exists a point of $U$ which is $T$-represented by the digit vector obtained by the placement of $\alpha^{(i)}$ on $B_{j_i}$ for $1\leq i\leq r$ and zero elsewhere.
\end{lemma}
\begin{proof}
For $q=(m-\normd)/a\in V$, we have $(m-1)/a<q\leq m/a$. As a result, $q\leq R_*$ if and only if $m\leq aR_*$. Hence the restriction of the given QTR function to $(-\infty,R_*]$ is again QTR with the same data, with support $U$; in particular, inserting multiples of $N$ zeros into long gaps, which leaves $m$ unchanged, maps points of $U$ to points of $U$. By \Cref{lem:long-gap-criterion} applied to $U$ and to $V$, some point $q_0=(m_0-\normd)/a$ of $U$ has at least $r$ long gaps, and no point of $V$ has more than $r$; thus $q_0$ has exactly $r$ long gaps. Let $\xi_1<\cdots<\xi_\rho$ be the nonzero positions of $\uld$, let $\kappa_1<\cdots<\kappa_r$ be the indices such that the gap preceding $\xi_{\kappa_i}$ is long, and put $\kappa_{r+1}=\rho+1$ and $\xi_*=\xi_{\kappa_1}$. We set
$$h\coloneqq\max\{cM,\xi_\rho\},\  s\coloneqq\xi_\rho-\xi_*+1,$$
and for $1\leq i\leq r$ we let $\alpha^{(i)}$ be the digit pattern of length $s$ whose entry in the coordinate $\xi_k-\xi_*+1$ is $d_{\xi_k}$ for $\kappa_i\leq k<\kappa_{i+1}$, and whose other entries are zero. In other words, $\alpha^{(i)}$ records the digits of $\uld$ lying between the $i$-th and the $(i+1)$-th long gap, at their original positions relative to $\xi_*$; in particular each $\alpha^{(i)}$ is nonzero. Finally, let $\Delta$ be any positive multiple of $N$ satisfying the inequality \eqref{eq:block-spacing}, and let $\beta\in\{1,\cdots,N\}$ be congruent to $\xi_*-h$ modulo $N$.

Now let integers $1\leq j_1<\cdots<j_r$ be given. Insert
$$h+\beta+j_1\Delta-\xi_*,\ (j_2-j_1)\Delta,\ \cdots,\ (j_r-j_{r-1})\Delta$$
zeros into the first, second, ..., $r$-th long gap of $\uld$, respectively. These numbers are nonnegative multiples of $N$ by the choice of $\beta$ and $\Delta$; hence the resulting digit vector $\uld'$ satisfies $q'\coloneqq(m_0-\norm{\uld'})/a\in U$. Let us locate the digits of $\uld'$. The digits $d_{\xi_k}$ with $k<\kappa_1$ are not moved and stay in positions at most $\xi_\rho\leq h$. For $\kappa_i\leq k<\kappa_{i+1}$, the digit $d_{\xi_k}$ is moved by the first $i$ insertions, i.e. by $h+\beta+j_i\Delta-\xi_*$ positions, to the position $h+\beta+j_i\Delta+(\xi_k-\xi_*)$. Note that deleting the first $h$ digits of $\uld'$ yields the digit vector $T$-representing $q'$, and the resulting digit vector is precisely the digit vector with $\alpha^{(i)}$ placed on $B_{j_i}$ for all $i$ and zeros elsewhere.
\end{proof}

The above discussion has direct consequences for our $\bfQ_p$-algebraic $p$-adic Hahn series $f$:
\begin{enumerate}
\item By \Cref{lem:tail-clusters}, the upper bound $\opn{ot}(\opn{Supp}(f_{\leq B}))<\omega^{\bfr(f)+1}$ implies that the support of every $\uld$ with $A_{\uld}\neq0$ form at most $\bfr(f)$ groups with diameter bounded by $cM$. 
\item By \Cref{lem:separated-placements}, the lower bound $\opn{ot}(\opn{Supp}(f_{\leq\bfB(f)}))\geq\omega^{\bfr(f)}$ provides a digit vector with exactly $\bfr(f)$ such groups, of valuation less than $\bfB(f)+1$ by the inequality \eqref{eq:small-truncation-coefficient}, whose groups can moreover be moved freely along a fixed grid of blocks. These blocks will be so far apart that neither lower powers of $g$ nor carries can interfere (cf. \Cref{lem:top-block-carry-free}), which is exactly the $(g,R,T,n)$-rigid condition.\end{enumerate}

\subsection{From separated blocks to interleavings}
\label{sec:block-coefficients}

We now combine the algebraic device of \Cref{sec:coefficient-reduction} with the combinatorics of \Cref{sec:digit-blocks}. Since all the constructions of this subsection and the next take place in one and the same situation, namely the one arising in the proof of \Cref{prop:uniform-polynomial-bound}, we state it once and for all.

\begin{assumption}\label{ass:standing}
The integers $n,r\geq1$ and $0\leq R_*\leq R$ are given, and $g\in\bfL_p$ is an element whose coefficient function $\opn{Coeff}(g,\cdot)$ is QTR with data $(a,b,c,M,N)$ and whose support $V\coloneqq\opn{Supp}(g)$ satisfies
$V\subseteq[0,R]$, $\opn{ot}(V)<\omega^{r+1}$, and $\opn{ot}(V\cap[0,R_*])\geq\omega^r$.
The integers $h,s,\beta,\Delta$, the nonzero digit patterns $\alpha^{(1)},\cdots,\alpha^{(r)}$ of length $s$ and the blocks $B_j=[\beta+j\Delta,\beta+j\Delta+s-1]_{\bfZ}$ ($j\geq1$) are those provided by \Cref{lem:separated-placements} applied to $V$, $R_*$ and $n$, and $T=ap^h$.
\end{assumption}

Under \Cref{ass:standing}, we have $V\subseteq S_{a,b,c}\subseteq a^{-1}\bfZ[1/p]\subseteq T^{-1}\bfZ[1/p]$, so the family $(A_{\uld})_{\uld\in\bbP}$ in the sum \eqref{eq:aggregate-valuation} is defined. Two facts will be used repeatedly. First, by the inequality \eqref{eq:block-spacing}, the blocks $B_j$ are pairwise disjoint and we have
\begin{equation}\label{eq:block-gap}
 \min B_{j+1}-\max B_j=\Delta-(s-1)>n+2cM+M,\ j=1,2,\cdots.
\end{equation}
Second, if $A_{\uld}\neq0$, then there exists a point of $V$ that is $T$-represented by $\uld$. So by \Cref{lem:tail-clusters} the support of $\uld$ is the union of at most $r$ nonempty sets of diameter at most $cM$.

\begin{lemma}\label{lem:top-block-carry-free}
Under \Cref{ass:standing}, let $\ulu\in\bbP$ be a digit vector whose support is contained in the union of exactly $nr$ of the blocks $B_j$ and meets each of them. Then $\ulu$ satisfies the $(g,R,T,n)$-rigid condition (cf. \Cref{def:grtn}). Moreover, if $\uld_1,\cdots,\uld_n\in\bbP$ satisfy $A_{\uld_i}\neq0$ for all $i$ and $\uld_1+\cdots+\uld_n=\ulu$, then each $\uld_i$ is supported on exactly $r$ of these $nr$ blocks, meeting each of them, and different $\uld_i$ occupy different blocks.
\end{lemma}
\begin{proof}
Let $0\leq k\leq n$, and let $\uld_1,\cdots,\uld_k\in\bbP$ satisfy $A_{\uld_i}\neq0$ for all $i$ and $\tau(\uld_1+\cdots+\uld_k)=\ulu$. We have to show that $k=n$ and $\uld_1+\cdots+\uld_k=\ulu$; the assertion on the supports will be obtained along the way. Since $\ulu\neq0$, we have $k\geq1$. Put $\uly\coloneqq\uld_1+\cdots+\uld_k\in\bigoplus_{\bfZ_{\geq1}}\bfN$, with coordinates $y_i$, and choose an integer $K$ such that $y_i=0$ for $i>K$. The digits of $\tau(\uly)=\ulu$ are obtained from those of $\uly$ by carrying: setting $e_i\coloneqq0$ for $i>K$ and, for $i=K,K-1,\cdots,1$,
$$e_i\coloneqq\left\lfloor\frac{y_i+e_{i+1}}{p}\right\rfloor,\  u_i\coloneqq y_i+e_{i+1}-pe_i,$$
one has $0\leq u_i<p$ and $\sum_{i=1}^Ku_ip^{-i}=\norm{\uly}-e_1$, so that $(u_1,\cdots,u_K,0,\cdots)\in\bbP$ has digit value congruent to $\norm{\uly}$ modulo $\bfZ$ and therefore equals $\tau(\uly)=\ulu$ by \Cref{lem:20585}. Note that the two displayed formulas remain valid for $i>K$ if we put $u_i\coloneqq0$ there, so we may use them for all $i\geq1$. Since $y_i\leq k(p-1)$ for all $i$, a descending induction on $i$ shows that $0\leq e_i\leq k-1$ for all $i$.

For any integer $j\geq 1$, let $B^+_j\coloneqq[\min B_j,\max B_j+n]_{\bfZ}$ and $B_j^{++}\coloneqq [\min B_j-cM,\max B_j+n+cM]_{\bfZ}$ associated to the block $B_j$. By the inequality \eqref{eq:block-gap}, $B_j^+$ (resp. $B_j^{++}$) are pairwise disjoint. Let $B$ be one of the $nr$ blocks occupied by $\ulu$.
We claim that $y_i\neq0$ for some $i\in B^+$. Suppose not, and let $i_0$ be the largest element of $B$ with $u_{i_0}\neq0$. Then $y_i=0$ for $i_0\leq i\leq i_0+n$, and also $u_i=0$ for $i_0<i\leq i_0+n$: for $i\leq\max B$ by the choice of $i_0$, and for $\max B<i\leq i_0+n\leq\max B+n$ because such $i$ lie strictly between $B$ and the next block by the inequality \eqref{eq:block-gap}, whereas $\ulu$ is supported on blocks. The recursion at $i=i_0+1,\cdots,i_0+n$ therefore reads $e_{i+1}=pe_i$, whence $e_{i_0+n+1}=p^ne_{i_0+1}$. As $e_{i_0+n+1}\leq k-1\leq n-1<p^n$, this forces $e_{i_0+1}=0$, and then the recursion at $i=i_0$ gives $u_{i_0}=-pe_{i_0}\leq0$, a contradiction.

Now, by the second fact following \Cref{ass:standing}, the support of each $\uld_i$ is the union of at most $r$ nonempty sets of diameter at most $cM$; we choose such a decomposition for each $i$ and call the resulting sets, at most $kr$ in number, the \textbf{pieces} with respect to $\uld_1,\cdots,\uld_k$. By the inequality \eqref{eq:block-gap}, the sets $\{B_j^+\}_{1\leq j\leq nr}$ attached to the $nr$ occupied blocks by $\ulu$ are pairwise at distance greater than $2cM$, so a piece meets at most one of them, and if a piece meets $B^+$, then it is contained in $B^{++}$. Since each of the $nr$ sets $B^+$ meets some piece, there must be exactly $nr$ pieces, each meeting exactly one set $B^+$ and contained in $B^{++}$; hence $k=n$, each $\uld_i$ has exactly $r$ pieces, and the pieces are in bijection with the occupied blocks. In particular, since $B_j^{++}$ are pairwise disjoint, pieces are pairwise disjoint. Consequently, at every position at most one of $\uld_1,\cdots,\uld_n$ has a nonzero digit, so that $0\leq y_i<p$ for all $i$; thus $\uly\in\bbP$ and $\uly=\tau(\uly)=\ulu$. This proves that $\uld$ satisfies the $(g,R,T,n)$-rigid condition.

Finally, let $\uld_1,\cdots,\uld_n$ be as in the last assertion of the lemma. Then $\tau(\uld_1+\cdots+\uld_n)=\tau(\ulu)=\ulu$, so everything above applies with $k=n$. Since all digits are nonnegative, $\uld_1+\cdots+\uld_n=\ulu$ implies that each $\uld_i$ is supported in the union of the occupied blocks. The piece attached to an occupied block $B$ is contained in $[\min B-cM,\max B+n+cM]_{\bfZ}$, which by the inequality \eqref{eq:block-gap} meets no occupied block other than $B$; hence this piece is contained in $B$. Therefore each $\uld_i$, whose $r$ pieces are attached to $r$ distinct blocks, is supported on exactly these $r$ blocks and meets each of them, and the blocks attached to different $\uld_i$ are different.

\end{proof}

\begin{lemma}\label{lem:block-position-invariance}
Under \Cref{ass:standing}, let $\gamma^{(1)},\cdots,\gamma^{(r)}$ be nonzero digit patterns of length $s$. For integers $1\leq j_1<\cdots<j_r$, let $\uld$ be the digit vector obtained by placing the digit patterns $\{\gamma^{(i)}\}$ on $\{B_{j_i}\}$ for $1\leq i\leq r$ and zeros elsewhere. Then $A_{\uld}$ does not depend on the choice of $j_1,\cdots,j_r$.
\end{lemma}
\begin{proof}
It suffices to compare an arbitrary choice $1\leq j_1<\cdots<j_r$ with the choice $j_i=i$ for all $i$. Let $\uld^*$ correspond to the latter and $\uld'$ to the former. Then $\uld'$ is obtained from $\uld^*$ by inserting $\Delta(j_1-1)$ zeros immediately before the block $B_1$, and then $\Delta(j_i-j_{i-1}-1)$ zeros immediately before the block carrying $\gamma^{(i)}$, for $i=2,\cdots,r$ successively. Note that all these numbers are nonnegative multiples of $N$.

For any integer $l$ and $\uld\in\bbP$, let
$$A_{\uld,l}\coloneqq \left[\opn{Coeff}\left(g,\frac{l-\norm{\uld}}{T}\right)\right]p^{l/T}.$$
Then $A_{\uld}=\sum_{l=0}^{TR} A_{\uld,l}$, and we only need to prove that $A_{\uld^*,l}=A_{\uld',l}$ for every $l=0,\cdots,TR$. Fix such $l$ and put $m\coloneqq\lceil l/p^h\rceil\geq0$, so that $0\leq(p^hm-l)/p^h<1$. Let $\ule^*\in\bbP$ be the digit vector whose first $h$ digits are those of the digit vector $1$-representing $(l-p^hm)/p^h$, and whose digit in the position $h+i$ is $d_i^*$ for $i\geq1$. We define $\ule'$ in the same way from $\uld'$. Then one can calculate that
\begin{equation}\label{eq:5468}
\frac{l-\norm{\uld^*}}{T}=\frac{m-\norm{\ule^*}}{a} \text{ and } \frac{l-\norm{\uld'}}{T}=\frac{m-\norm{\ule'}}{a}.
\end{equation}

Note that $\Psi(\ule^*)=\Psi(\ule')$. %
If $\Psi(\ule^*)>c$, since $\ule^*$ (resp. $\ule'$) is the unique digit vector that $a$-represents $(m-\norm{\ule^*})/a$ (resp. $(m-\norm{\ule'})/a$), we know that $(m-\norm{\ule^*})/a$ and $(m-\norm{\ule'})/a$ do not belong to $S_{a,b,c}\supseteq V$. The equalities \eqref{eq:5468} implies that $A_{\uld^*,l}=A_{\uld',l}=0$.

Now we suppose $\Psi(\ule^*)\leq c$. The consecutive $\beta+\Delta-1$ digits of $\ule^*$ preceding the shifted block $B_1+h$ are zeros, and so do the $\Delta-s$ digits between two consecutive shifted blocks. Note that both $\beta+\Delta-1$ and $\Delta-s$ are at least $M$ by inequality \eqref{eq:block-spacing}, and $\ule'$ is obtained from $\ule^*$ by the same insertions of zeros as $\uld'$ from $\uld^*$, shifted by h positions. The element $\ule'$ is thus obtained from $\ule^*$ by successively inserting $N$ zeros, each time at a position immediately preceded by at least $M$ consecutive zeros. As a result, the QTR identity implies $\opn{Coeff}(g,(m-\norm{\ule^*})/a)=\opn{Coeff}(g,(m-\norm{\ule'})/a)$. The result follows from the equalities \eqref{eq:5468} again.
\end{proof}

\begin{definition}\label{def:words}
Let $\calL$ be a finite nonempty set, called an \textbf{alphabet}. For an integer $d\geq0$, a \textbf{word} of length $d$ over $\calL$ is a tuple $w=(w_1,\cdots,w_d)\in\calL^d$, whose entries are called its \textbf{letters}. For $I=\{i_1<\cdots<i_e\}\subseteq\{1,\cdots,d\}$, we write $w|_I\coloneqq(w_{i_1},\cdots,w_{i_e})$, a word of length $e$ (the empty word if $I=\varnothing$). If $z$ is a word of length $nr$ and $(I_1,\cdots,I_n)$ is an ordered partition of $\{1,\cdots,nr\}$ into $n$ sets of cardinality $r$, we say that $z$ is an \textbf{interleaving} of the words $z|_{I_1},\cdots,z|_{I_n}$, and that the ordered partition $(I_1,\cdots,I_n)$ realizes it as such.
\end{definition}

By \Cref{lem:block-position-invariance}, the coefficient $A_{\uld}$ of a digit vector built from $r$ patterns depends only on the sequence of patterns, that is, on a word in the patterns; and \Cref{lem:top-block-carry-free} shows that the products contributing to $\calS_n(\ulu)$ for $\ulu$ on $nr$ blocks correspond to ways of splitting a word of length $nr$ into $n$ words of length $r$. 

We apply this to the alphabet $\calL$ of all nonzero digit patterns of length $s$. For a word $w=(w_1,\cdots,w_r)\in\calL^r$, we denote by $A_w$ the common value, by \Cref{lem:block-position-invariance}, of $A_{\uld}$ for the digit vectors $\uld$ with $w_i$ placed on $B_{j_i}$ for $1\leq i\leq r$ and zeros elsewhere, where $1\leq j_1<\cdots<j_r$ are arbitrary. For a word $z=(z_1,\cdots,z_{nr})\in\calL^{nr}$, we denote by $\ulu_z\in\bbP$ the digit vector with $z_j$ placed on $B_j$ for $1\leq j\leq nr$ and zeros elsewhere; its support is contained in $B_1\cup\cdots\cup B_{nr}$ and meets each of these blocks, so that \Cref{lem:top-block-carry-free} applies to $\ulu_z$.

\begin{lemma}\label{lem:block-interleaving-identity}
Under \Cref{ass:standing}, for every $z\in\calL^{nr}$ one has
\begin{equation}\label{eq:block-interleaving-identity}
 \calS_n(\ulu_z)=\sum_{(I_1,\cdots,I_n)}\prod_{i=1}^nA_{z|_{I_i}},
\end{equation}
where the sum runs over all ordered partitions $(I_1,\cdots,I_n)$ of $\{1,\cdots,nr\}$ into $n$ sets of cardinality $r$.
\end{lemma}
\begin{proof}
Let $(\uld_1,\cdots,\uld_n)$ be a tuple with $\uld_1+\cdots+\uld_n=\ulu_z$ and $A_{\uld_i}\neq0$ for all $i$. Such a tuple corresponds to a nonzero term in the summation in equality \eqref{eq:carry-free-coefficient}, i.e., the definition of $\calS_n(\ulu_z)$. By \Cref{lem:top-block-carry-free}, each $\uld_i$ is supported on exactly $r$ of the blocks $B_1,\cdots,B_{nr}$, meeting each of them, and different $\uld_i$ occupy different blocks. Since every block $B_j$ carries a nonzero digit of $\ulu_z=\uld_1+\cdots+\uld_n$, the sets $I_i\coloneqq\{j\colon\opn{supp}(\uld_i)\cap B_j\neq\varnothing\}$ form an ordered partition $(I_1,\cdots,I_n)$ of $\{1,\cdots,nr\}$ into sets of cardinality $r$. Thus $(I_1,\cdots,I_n)$ realizes $z$ as an interleaving of the words $z|_{I_1},\cdots,z|_{I_n}$, where $z|_{I_i}$ records the patterns
of $\uld_i$ on its occupied blocks in increasing order.
Hence $A_{\uld_i}=A_{z|_{I_i}}$. Conversely, for every ordered partition $(I_1,\cdots, I_n)$ as in the statement and any $i=1,\cdots,n$, placing the letters of $z|_{I_i}$ on the blocks $B_j$ for every $j\in I_i$ defines digit vectors $\uld_i$, and one has $\uld_1+\cdots+\uld_n=\ulu_z$. Since distinct partitions $(I_1,\cdots, I_n)$ give distinct tuples $(\uld_1,\cdots,\uld_n)$, the two sides of the equality \eqref{eq:block-interleaving-identity} therefore have the same nonzero terms.
\end{proof}

\subsection{A uniformly bounded nonzero coefficient}
\label{sec:bounded-coefficient}

By \Cref{lem:block-interleaving-identity}, the coefficient $\calS_n(\ulu_z)$ is a sum, indexed by interleavings, of products of the quantities $A_w$, $w\in\calL^r$. The following lemma, a statement of pure algebra, shows that cancellation in the summation $\calS_n(\ulu_z)$ cannot occur for all $z$ simultaneously, and that it costs, in terms of valuation, at most a constant depending on $n$ and $r$ but not on the alphabet.

\begin{lemma}\label{lem:shuffle-power-coefficient}
Let $n,r\geq1$ be integers. Let $\calL$ be a finite nonempty alphabet and let $(\lambda_w)_{w\in\calL^r}$ be a family of elements of $\breve{\bfZ}_{p,T}$, not all zero. For $z\in\calL^{nr}$ put
$$\sigma_z\coloneqq\sum_{(I_1,\cdots,I_n)}\prod_{i=1}^n\lambda_{z|_{I_i}},$$
where the sum runs over all ordered partitions of $\{1,\cdots,nr\}$ into $n$ sets of cardinality $r$.  There exists an integer $C(n,r)\geq0$, depending only on $n$, $r$ and $p$, and a word $z\in\calL^{nr}$ of length $nr$, such that
$$\sigma_z\neq0\ \text{and}\  v_p(\sigma_z)\leq n\cdot\min\{v_p(\lambda_w)\colon w\in\calL^r,\ \lambda_w\neq0\}+C(n,r).$$
\end{lemma}
\begin{proof}
Choose $w_0\in\calL^r$ with $\lambda_{w_0}\neq0$ and $v_p(\lambda_{w_0})$ minimal, and let $\calL_0\subseteq\calL$ be the set of letters occurring in $w_0$, so that $\calL_0$ has at least one and at most $r$ elements. For $z\in\calL_0^{nr}$, every word $z|_{I_i}$ lies in $\calL_0^r$, so $\sigma_z$ only involves the subfamily $(\lambda_w)_{w\in\calL_0^r}$, whose minimal valuation among nonzero members is still $v_p(\lambda_{w_0})$; we may therefore replace $\calL$ by $\calL_0$. Next, if $\calL_0$ has fewer than $r$ elements, we enlarge $\calL_0$ to a set $\calL_1$ with exactly $r$ elements and set $\lambda_w\coloneqq0$ for every $w\in\calL_1^r\setminus\calL_0^r$. Then $\sigma_z=0$ for every $z\in\calL_1^{nr}\setminus\calL_0^{nr}$, since every ordered partition assigns a letter of $\calL_1\setminus\calL_0$ to one of the words $z|_{I_i}$, and the minimal valuation is unchanged. Altogether, we may and do assume that $\calL=\{1,\cdots,r\}$. We order the words of a given length over $\calL$ lexicographically.

Let $X=(X_w)_{w\in\calL^r}$ be a family of indeterminates and, for $z\in\calL^{nr}$, let
$$\Sigma_z(X)\coloneqq\sum_{(I_1,\cdots,I_n)}\prod_{i=1}^nX_{z|_{I_i}}\in\bfZ[X],$$
a homogeneous polynomial of degree $n$ with nonnegative integer coefficients. We claim that the polynomials $\Sigma_z$, $z\in\calL^{nr}$, have no common zero in $\overline{\bfQ}^{\,\calL^r}$ other than the origin. Let indeed $x=(x_w)_{w\in\calL^r}$ be a nonzero point of $\overline{\bfQ}^{\,\calL^r}$, let $w$ be the lexicographically largest word with $x_w\neq0$, and let $z$ be the lexicographically largest interleaving of $n$ copies of $w$. Consider an ordered partition $(I_1,\cdots,I_n)$ whose term $\prod_ix_{z|_{I_i}}$ is nonzero. By the choice of $w$, we have $z|_{I_i}\leq w$ for all $i$. If $z|_{I_i}<w$ for some $i$, then replacing the letters of $z$ in the positions $I_i$ by those of $w$ produces a word $z'>z$, because the first modified position carries a larger letter; replacing successively the letters in the other positions $I_{i'}$ by those of $w$ can only increase the word further, and ends with an interleaving of $n$ copies of $w$ that is larger than $z$, a contradiction. Hence $z|_{I_i}=w$ for all $i$, and $\Sigma_z(x)$ equals $x_w^n$ times the number of ordered partitions realizing $z$ as an interleaving of $n$ copies of $w$, which is a positive integer. Thus $\Sigma_z(x)\neq0$, proving the claim.

By Hilbert's Nullstellensatz (cf. \cite[Tag 00FV]{stacks-project}), the radical of the ideal of $\bfQ[X]$ generated by the $\Sigma_z$'s is therefore the maximal ideal generated by the $X_w$'s. Hence there exists an integer $e\geq1$ such that each $X_w^e$ is a $\bfQ[X]$-linear combination of the $\Sigma_z$, and clearing denominators we find an integer $D\geq1$ and polynomials $Q_{w,z}\in\bfZ[X]$ such that
\begin{equation}\label{eq:nullstellensatz-identity}
 D\,X_w^e=\sum_{z\in\calL^{nr}}Q_{w,z}(X)\,\Sigma_z(X)\ \text{for every }w\in\calL^r.
\end{equation}
We fix such $e$, $D$ and $Q_{w,z}$ once and for all for each pair $(n,r)$, and put $C(n,r)\coloneqq v_p(D)$.

Put $\mu_w\coloneqq\lambda_w/\lambda_{w_0}$ for $w\in\calL^r$ and $\mu\coloneqq(\mu_w)_{w\in\calL^r}$. By the choice of $w_0$, every $\mu_w$ has nonnegative valuation, and $\mu_{w_0}=1$. Evaluating the equality \eqref{eq:nullstellensatz-identity} for $w=w_0$ at $\mu=(\mu_w)$ gives 
$$D=\sum_zQ_{w_0,z}(\mu)\Sigma_z(\mu).$$ 
Since $v_p(Q_{w_0,z}(\mu))\geq0$ for all $z$, the ultrametric inequality implies that there exists a $z$ such that $v_p(\Sigma_z(\mu))\leq v_p(D)=C(n,r)$; in particular $\Sigma_z(\mu)\neq0$. As the polynomial $\Sigma_z$ is homogeneous of degree $n$, we have $\sigma_z=\Sigma_z(\lambda)=\lambda_{w_0}^n\Sigma_z(\mu)$, whence $\sigma_z\neq0$ and $$v_p(\sigma_z)=n\,v_p(\lambda_{w_0})+v_p(\Sigma_z(\mu))\leq n\,v_p(\lambda_{w_0})+C(n,r).$$
\end{proof}

We can now construct the digit vector that satisfies the condition in \Cref{lem:truncation-coset}.  %

\begin{lemma}\label{lem:bounded-separated-coefficient}
Under \Cref{ass:standing}, there exists $\ulu\in\bbP$ satisfying the $(g,R,T,n)$-rigid condition such that
$$\calS_n(\ulu)\neq0\ \text{and}\  v_p(\calS_n(\ulu))\leq n(R_*+1)+C(n,r),$$
where $C(n,r)$ is the constant of \Cref{lem:shuffle-power-coefficient}.
\end{lemma}
\begin{proof}
Let $\calL$ be the alphabet of nonzero digit patterns of length $s$, and consider the family $(A_w)_{w\in\calL^r}$ introduced after \Cref{def:words}. Let $w_*\coloneqq(\alpha^{(1)},\cdots,\alpha^{(r)})\in\calL^r$ be the word formed by the patterns of \Cref{ass:standing}. By \Cref{lem:separated-placements}, the digit vector with $\alpha^{(i)}$ placed on $B_i$ for $1\leq i\leq r$ $T$-represents a point of $V\cap(-\infty,R_*]=V\cap[0,R_*]$, so the inequality \eqref{eq:small-truncation-coefficient} gives $A_{w_*}\neq0$ and $v_p(A_{w_*})<R_*+1$. In particular, the family $(A_w)_{w\in\calL^r}$ is not identically zero, and the minimum of $v_p(A_w)$ over the words $w$ with $A_w\neq0$ is less than $R_*+1$. \Cref{lem:shuffle-power-coefficient}, applied with $\lambda_w=A_w$, provides $z\in\calL^{nr}$ such that
$$\sigma_z=\sum_{(I_1,\cdots,I_n)}\prod_{i=1}^nA_{z|_{I_i}}$$
is nonzero of valuation at most $n(R_*+1)+C(n,r)$. Put $\ulu\coloneqq\ulu_z$. Then $\calS_n(\ulu)=\sigma_z$ by \Cref{lem:block-interleaving-identity}, and $\ulu$ satisfies the $(g,R,T,n)$-rigid condition by \Cref{lem:top-block-carry-free}.
\end{proof}

Now we are ready to prove \Cref{prop:uniform-polynomial-bound}.
\begin{proof}[Proof of \Cref{prop:uniform-polynomial-bound}]
Let $f$, $P$ and $B$ be as in the statement, and put
$$n\coloneqq\deg(P),\  r\coloneqq\bfr(f),\  g\coloneqq f_{\leq B},\  V\coloneqq\opn{Supp}(g),\  R\coloneqq B,\  R_*\coloneqq\bfB(f),$$
and let $C(n,r)$ be the constant of \Cref{lem:shuffle-power-coefficient}. 
We check that \Cref{ass:standing} is satisfied. Since $f$ is integral and nonzero, $v_p(f)\geq0$, so that $V\subseteq[0,B]$ and $\bfB(f)\geq v_p(f)\geq0$ by \Cref{def:essential-truncation}; together with the hypothesis $B>\bfB(f)$, this gives $0\leq R_*\leq R$. The coefficient function of $g=f_{\leq B}$ is QTR by \Cref{lem:local-qtr-full}. Finally, $\opn{ot}(V)<\omega^{r+1}$ by \Cref{lem:14033}, and $V\cap[0,R_*]=\opn{Supp}(f_{\leq\bfB(f)})$ has order type at least $\omega^r$ by \Cref{def:essential-truncation}.

By \Cref{lem:bounded-separated-coefficient}, there exists $\ulu\in\bbP$ satisfying the $(g,R,T,n)$-rigid condition with $\calS_n(\ulu)\neq0$ and $v_p(\calS_n(\ulu))\leq n(\bfB(f)+1)+C(n,r)$. And the result follows from applying \Cref{lem:truncation-coset} to the monic polynomial $P\in\bfZ_p[X]$:%
$$v_p(P(f_{\leq B}))=v_p(P(g))\leq v_p(\calS_n(\ulu))\leq n(\bfB(f)+1)+C(n,r),$$
\end{proof}

\section{Main theorem and its application}\label{sec:main}
\subsection{Finite-rank exclusion}\label{sec:finite-rank}

We apply the uniform estimate of $v_p(P(f_{\leq B}))$ to a growing truncation of a
$\bfQ_p$-algebraic element $f\in\bfL_p$. A smaller truncation fixes the upper bound; the polynomial
relation forces the valuation to exceed that bound.

\begin{theorem}\label{lem:finite-rank-exclusion}
Let $f\in\bfL_p$ be algebraic over $\bfQ_p$. If
$\operatorname{ot}(\operatorname{Supp}(f))<\omega^\omega$, then
$\operatorname{ot}(\operatorname{Supp}(f))\leq\omega$.
\end{theorem}
\begin{proof}
The case $f=0$ is immediate. Multiplying a nonzero $f$ by a sufficiently
large integral power of $p$ translates its support into $[0,\infty)$
without changing its order type or algebraicity. We may therefore
assume $f\in\calO_{\overline{\bfQ}_p}$.
Its monic minimal polynomial $P\in\bfQ_p[X]$ then belongs to
$\bfZ_p[X]$.

If $\bfr(f)=0$, the conclusion follows from \Cref{lem:14033}.
Otherwise, choose an integer $B$ such that
$$B>\deg(P)\cdot(\bfB(f)+1)+C(\deg(P),\bfr(f)),$$ 
where $C(\cdot,\cdot)$ is the function in \Cref{prop:uniform-polynomial-bound}.
Since $\bfB(f)\geq0$ and $C(\deg(P),\bfr(f))\geq0$, we have $B>\bfB(f)$.

By \Cref{prop:uniform-polynomial-bound}, we have 
\begin{equation}\label{eq:29507}
 v_p(P(f_{\leq B}))\leq \deg(P)(\bfB(f)+1)+C(\deg(P),\bfr(f))<B.
\end{equation}
On the other hand, write
$P(X)-P(Y)=(X-Y)Q(X,Y)$ with $Q\in\bfZ_p[X,Y]$.
Both $f$ and $f_{\leq B}$ are integral, and $P(f)=0$, whence
\[
 v_p(P(f_{\leq B}))
 =v_p\bigl((f_{\leq B}-f)Q(f_{\leq B},f)\bigr)
 \geq v_p(f_{\leq B}-f)>B.
\]
This contradicts the inequality \eqref{eq:29507}, proving the theorem.
\end{proof}

Together with \Cref{prop:external-bounds}, this proves \Cref{thm:main}. As a quick corollary, \Cref{lem:finite-rank-exclusion} implies the following consequence,
which was previously discussed in \cite{wang2026padichahnseriessparse}.%
\begin{proposition}\label{prop:19985}
    Let $f$ be a $\bfQ_p$-algebraic element in $\bfL_p$. If the support of $f$ is bounded, then the support of $f$ is finite.
\end{proposition}
\begin{proof}
    Suppose $f\in\bfL_p$ is $\bfQ_p$-algebraic and $\opn{Supp}(f)\subseteq (-\infty,M]$ for some $M\in\bfQ$. Then $f+p^{M+1}$ is also $\bfQ_p$-algebraic, and $\opn{ot}(\opn{Supp}\left(f+p^{M+1}\right))$ is a successor ordinal. By \Cref{prop:external-bounds}, this successor ordinal is at most $\omega^\omega$, hence strictly less than $\omega^\omega$. By \Cref{lem:finite-rank-exclusion}, it is at most $\omega$; being a successor, it must be finite. Hence $\opn{Supp}\left(f+p^{M+1}\right)$ is finite, and so is $\opn{Supp}(f)$.
\end{proof}

By Krasner's lemma, we have an explicit description of $\bfQ_p$-algebraic elements with small support ordinals.
\begin{proposition}\label{prop:57854}
Let $f\in\bfL_p$ be algebraic over $\bfQ_p$. Then
$\operatorname{ot}(\operatorname{Supp}(f))\leq\omega$ if and only if
\[
 f\in\bigcup_{T=1}^{\infty}\bfQ_p^{\operatorname{un}}
                  \left(p^{1/T}\right).
\]
\end{proposition}
\begin{proof}
Let $F\coloneqq \bigcup_{T=1}^{\infty}\bfQ_p^{\operatorname{un}}\left(p^{1/T}\right)$. Note that $$\bfQ_p^{\operatorname{un}}
 =\bigcup_{s\geq1}\operatorname{Frac}W(\bfF_{p^s})
,$$
 an element of $F$ belongs to
$\operatorname{Frac}W(\bfF_{p^s})(p^{1/T})$ for some integers $s,T\geq1$.
This finite extension of $\bfQ_p$ has uniformizer $p^{1/T}$ and residue field
$\bfF_{p^s}$, so its elements have standard expansions
\[
 \sum_{i=i_0}^{\infty}[a_i]p^{i/T},\  i_0\in\bfZ,\quad a_i\in\bfF_{p^s}.
\]
Their supports have order type at most $\omega$.

Conversely, it makes no harm to assume that the support has order type $\omega$. By \Cref{prop:19985}, the support is unbounded, hence for every integer $m$, the bounded truncation $f_{\leq m}$ is finite, belongs to $F$, and $v_p(f-f_{\leq m})>m$. Since $f_{\leq m}$ tends to $f$ with $m\lto\infty$,  Krasner's lemma (cf. \cite[Chapter 3, Section 3.4.2]{boschNonArchimedeanAnalysisSystematic1984a}) implies that $\bfQ_p(f)\subseteq\bfQ_p(f_{\leq m})\subset F$, and the result follows.
\end{proof}

\subsection{Bi-algebraicity of the shadow map}
Recall the shadow map of \Cref{def:shadow-map} and the homeomorphism
of completed algebraic closures in \Cref{prop:25416}.
Note that the image of $\bfF_p\pparen{t}$ under $\Theta$ is exactly the field $\bfQ_p$. \Cref{prop:25416} shows that, up to taking completions, $\Theta$ preserves algebraicity over the respective fields $\bfF_p\pparen{t}$ and $\bfQ_p$. It is therefore natural to ask whether $\Theta$ preserves algebraicity without taking completions. 

In \cite{wang2025padictranscendencesumk1inftyp1pk}, we give a negative answer to this question in full generality by constructing an element $f\in\bfC_p\backslash\overline{\bfQ}_p$ such that $\Theta^{-1}(f)$ is algebraic over $\bfF_p\pparen{t}$. It remains to determine the exact set of elements $f\in\bfL_p^\flat$ such that $f$ is algebraic over $\bfF_p\pparen{t}$ and $\Theta(f)$ is algebraic over $\bfQ_p$. The following proposition gives a complete answer to this question.

\begin{proposition}\label{prop:bialgebraic}
    Let $f\in\bfL_p^\flat$ be a Hahn series. Then the following are equivalent:
    \begin{enumerate}
        \item $f$ is algebraic over $\bfF_p\pparen{t}$ and $\Theta(f)$ is algebraic over $\bfQ_p$;
        \item $\Theta(f)$ belongs to the field $\bigcup_{T=1}^\infty\bfQ_p^{\opn{un}}\left(p^{1/T}\right)$;
        \item $\Theta(f)$ is algebraic over $\bfQ_p$ and $\opn{ot}(\opn{Supp}(\Theta(f)))\leq \omega$.
    \end{enumerate}
\end{proposition}
\begin{proof}
    The equivalence of the second and the third assertions is exactly \Cref{prop:57854}. Now assume the first assertion. Lisinski proved\footnote{Although this result was first explicitly stated by Lisinski in \cite{lisinskiApproximationAlgebraicityPositive2023} in 2023, it is actually a quick corollary of \cite[Theorem 15]{kedlayaAlgebraicClosurePower2001} by Kedlaya in 2001. The $p$-adic analogue of this result (cf. \Cref{prop:external-bounds}) was asserted by Kedlaya in \cite[Section 4]{kedlayaPowerSeriesPAdic2001}.} in \cite[Theorem 4.5]{lisinskiApproximationAlgebraicityPositive2023} that $\opn{ot}(\opn{Supp}(f))=\opn{ot}(\opn{Supp}(\Theta(f)))$ is strictly less than $\omega^\omega$. Since $\Theta(f)$ is algebraic over $\bfQ_p$, \Cref{lem:finite-rank-exclusion} implies that $\opn{ot}(\opn{Supp}(\Theta(f)))\leq \omega$. Hence the first assertion implies the third one. 
    
    Conversely, if the second assertion holds, then $\Theta(f)$ can be written in the form
    $$\Theta(f)=\sum_{i=s}^\infty[f_i]p^{i/T},$$
    where $T>0$, $s$ are integers and there exists an integer $r\geq1$ such that $f_i\in \bfF_{p^r}$ for every $i\geq s$. As a result, $f=\sum_{i=s}^\infty f_i t^{i/T}$ lies in the field $\bfF_{p^r}\pparen{t^{1/T}}$, which is algebraic over $\bfF_p\pparen{t}$. Hence the second assertion implies the first one.
\end{proof}

\printbibliography
\end{document}